\documentclass[reqno]{amsart}
\usepackage{latexsym,amsmath,amssymb,amscd}
\usepackage[all]{xy}
\usepackage{enumerate}
\makeatletter
\@addtoreset{figure}{section}
\def\thefigure{\thesection.\@arabic\c@figure}
\def\fps@figure{h,t}
\@addtoreset{table}{bsection}

\def\thetable{\thesection.\@arabic\c@table}
\def\fps@table{h, t}
\@addtoreset{equation}{section}

\makeatother

\newtheorem{theorem}{Theorem}
\newtheorem{corollary}[theorem]{Corollary}

\newtheorem{example}[theorem]{Example}
\newtheorem{lemma}[theorem]{Lemma}

\newtheorem{proposition}[theorem]{Proposition}
\newtheorem{remark}[theorem]{Remark}

\numberwithin{theorem}{section}
\numberwithin{equation}{section}

\newcommand{\Ad}{{\rm Ad}}
\newcommand{\ad}{{\rm ad}}

\newcommand{\Ker}{{\rm Ker}\,}

\newcommand{\spa}{{\rm span}\,}

\newcommand{\RR}{{\mathbb R}}

\newcommand{\Bc}{{\mathcal B}}

\newcommand{\Hc}{{\mathcal H}}

\newcommand{\Oc}{{\mathcal O}}
\newcommand{\Pc}{{\mathcal P}}

\newcommand{\Vc}{{\mathcal V}}
\newcommand{\Xc}{{\mathcal X}}

\newcommand{\Zc}{{\mathcal Z}}

\newcommand{\ag}{{\mathfrak a}}

\newcommand{\hg}{{\mathfrak h}}
\newcommand{\mg}{{\mathfrak m}}
\renewcommand{\ng}{{\mathfrak n}}
\newcommand{\sg}{{\mathfrak s}}
\newcommand{\zg}{{\mathfrak z}}

\title[Coadjoint orbits of stepwise square integrable representations]{Coadjoint orbits of stepwise square integrable representations}

\author{Ingrid Belti\c t\u a}
\author{Daniel Belti\c t\u a}
\address{Institute of Mathematics ``Simion Stoilow'' of the Romanian Academy,
P.O. Box 1-764, Bucharest, Romania}

\email{Ingrid.Beltita@imar.ro, ingrid.beltita@gmail.com}
\email{Daniel.Beltita@imar.ro, beltita@gmail.com}

\begin{document}

\parskip=5pt

\begin{abstract}
Nilpotent Lie groups with 
stepwise square integrable representations 
were recently investigated by J.A.~Wolf. 
We give an alternative approach to these representations by relating them 
to the stratifications of the duals of nilpotent Lie algebras, 
thus proving that they correspond to a subset with relative Hausdorff topology, dense interior, 
and total Plancherel measure 
in the unitary dual of the Lie group under consideration. 
\end{abstract}

\maketitle

\section{Introduction}
\label{}

Nilpotent Lie groups with 
stepwise square integrable representations 
were recently investigated by J.A.~Wolf \cite{Wo13a}. 
We give an alternative approach to these representations, 
using the method of coadjoint orbits and thus relating them to the 
stratifications of duals of nilpotent Lie algebras 
\cite{Pd89}. 
This directly leads to basic properties of stepwise square integrable representations 
like the fact that they correspond 
to 
a subset with relative Hausdorff topology, dense interior, and total Plancherel measure  
in the unitary dual of the Lie group under consideration 
(Corollary~\ref{main4}).  
This also fits well with the name ``principal series'' that J. Dixmier \cite{Di59} earlier used for them 
in the special case of the strictly upper triangular groups.  

As proved in \cite{Wo13a}, examples of nilpotent Lie groups we consider include 
nilradicals of minimal Borel subgroups of semisimple Lie groups 
(Example 1 below). 
We point out here a completely different class of examples, 
namely all 3-step nilpotent Lie groups with 1-dimensional centers 
(see Example 2).  

\noindent {\bf Preliminaries}. 
Throughout this note, by nilpotent Lie group we actually mean a connected, simply connected, nilpotent Lie group. 
We denote Lie algebras by lower case Gothic letters and their corresponding Lie groups by upper case Roman letters. 
All nilpotent Lie algebras are assumed real and finite dimensional. 
The semidirect product notation $\widetilde{\ng}=\mg\ltimes\ng$ means that 
$\widetilde{\ng}$ is a Lie algebra with a vector direct sum decomposition $\widetilde{\ng}=\mg\dotplus\ng$, 
where $\ng$ is an ideal of $\widetilde{\ng}$ (that is, $[\widetilde{\ng},\ng]\subseteq\ng$) 
and $\mg$ is a subalgebra of $\widetilde{\ng}$.  
Similar notation is used for groups. 

Let $\ng$ be any nilpotent Lie algebra of dimension $m\ge 1$ 
with its coresponding nilpotent Lie group $N=(\ng,\cdot)$ 
whose multiplication is defined by the Baker-Campbell-Hausdorff fomula. 
Select any Jordan-H\"older basis 
$\{X_1,\dots,X_m\}\subset\ng$.   
So for $j=1,\dots,m$ if we define $\ng_j:=\spa\{X_1,\dots,X_j\}$ then $[\ng,\ng_j]\subseteq\ng_{j-1}$, 
where $\ng_0:=\{0\}$. 
A linear subspace $\Vc\subseteq\ng$ is called compatible with the above Jordan-H\"older basis 
if $\Vc\cap\{X_1,\dots,X_m\}$ is a basis in $\Vc$.  

Let $\ng^*$ be the dual space of $\ng$,  
with the duality pairing $\langle\cdot,\cdot\rangle\colon\ng^*\times\ng\to\RR$. 
The $N$-coadjoint action is $\Ad^*_N\colon N\times\ng^*\to\ng^*$, 
$(x,\xi)\mapsto\xi\circ\exp(-\ad_{\ng}x)$, 
where  $\ad_{\ng}x:=[x,\cdot]\colon\ng\to\ng$. 
The coadjoint isotropy algebra at any $\xi_0\in\ng^*$ is 
$\ng(\xi_0):=
\{x\in\ng\mid[x,\ng]\subseteq\Ker\xi_0\}$ and the  
jump set 
$J(\xi_0):=\{j\mid X_j\not\in\ng_{j-1}+\ng(\xi_0)\}$ depends only on the coadjoint orbit~$\Oc:=\Ad^*_N(N)\xi_0\subseteq\ng^*$,  
so we may denote $e(\Oc):=J(\xi_0)$. 
If one defines $\ng_e:=\spa\{X_j\mid j\in e\}$,  
one has the vector direct sum  decomposition $\ng=\ng(\xi_0)\dotplus\ng_e$ and the mapping $\Oc\to\ng_e^*$, $\xi\mapsto\xi\vert_{\ng_e}$, 
is a diffeomorphism \cite[Lemma 1.6.1]{Pd89}. 

For any integer $m\ge 1$ denote by $\Pc(m)$ the set of all subsets of $\{1,\dots,m\}$. 
For $e_1,e_2\in\Pc(m)$ one defines 
$e_1\prec e_2\iff \min(e_1\setminus e_2)<\min(e_2\setminus e_1)$, 
where $\min\emptyset:=\infty$. 
Then $\prec$ is a total ordering on $\Pc(m)$, 
for which the largest subset is $\emptyset$ and the smallest subset is $\{1,\dots,m\}$. 
For any nilpotent Lie algebra $\ng$ with a fixed Jordan-H\"older series,  
$e(\ng)$ denotes the $\prec$-minimum of all jump sets of coadjoint orbits of~$\ng$. 

\begin{proposition}[{\cite[Th. 2.2.7]{Pd94},\cite[Prop. 2.8]{BB11}}]\label{orthog}
With the above notation, consider any unitary irreducible representation $\pi\colon N\to\Bc(\Hc)$  
associated with the coadjoint orbit~$\Oc$. 
For arbitrary 
$f,\phi\in\Hc$     
one has  
$$\Vert(f\mid\pi(\cdot)\phi)\Vert^2_{L^2(\ng_e)} 
=C_e\Vert f\Vert^2_{\Hc}\cdot\Vert\phi\Vert^2_{\Hc}, $$ 
where $C_e:=(2\pi)^{(\dim\Oc)/2}/\vert{\rm Pf}_e(\xi)\vert$ depends only on the jump set~$e$.  
\end{proposition}

One says in this paper that a nilpotent Lie algebra $\ng$ has \emph{generic flat coadjoint orbits} if there exists $\xi\in\ng^*$ with $\ng(\xi)=\Zc(\ng):=\{x\in\ng\mid[x,\ng]=\{0\}\}$ 
(the center of $\ng$). 
If this the case, then the set of all $\xi\in\ng^*$ with $\ng(\xi)=\Zc(\ng)$ is an $\Ad^*_N$-invariant Zariski-dense open subset of $\ng^*$ 
and the coefficients of the unitary irreducible representations of $N$ associated with $N$-coadjoint orbits contained in that set are square integrable modulo the center of $N$; see \cite{MW73} and \cite{CG90} for more details. 

\section{Main results}

\begin{proposition}\label{grad}
If $\widetilde{\ng}=\mg\ltimes\ng$ 
and $\widetilde{\xi}\in\widetilde{\ng}^*$ with $[\mg,\ng]\subseteq\Ker\widetilde{\xi}$, 
then 
$$\widetilde{\ng}(\widetilde{\xi})=\mg(\widetilde{\xi}\vert_{\mg})\dotplus\ng(\widetilde{\xi}\vert_{\ng}).$$ 
\end{proposition}

\begin{proof}
Step 1. 
For arbitrary $\widetilde{\xi}\in\widetilde{\ng}^*$ one has 
$[\widetilde{\ng},\mg(\widetilde{\xi}\vert_{\mg})]
=[\mg,\mg(\widetilde{\xi}\vert_{\mg})]+[\ng,\mg(\widetilde{\xi}\vert_{\mg})]$ 
and the first term in the right-hand side of that equality is contained in $\Ker(\widetilde{\xi}\vert_{\mg})$, 
hence 
\begin{equation}\label{grad_proof_eq1}
\mg(\widetilde{\xi}\vert_{\mg})\subseteq\widetilde{\ng}(\widetilde{\xi})
\iff [\ng,\mg(\widetilde{\xi}\vert_{\mg})]\subseteq\Ker(\widetilde{\xi}\vert_{\ng}).
\end{equation}
Similarly, since 
$[\widetilde{\ng},\ng(\widetilde{\xi}\vert_{\ng})]
=[\mg,\ng(\widetilde{\xi}\vert_{\ng})]+[\ng,\ng(\widetilde{\xi}\vert_{\ng})]$ 
and the second term in the right-hand side of that equality is contained in $\Ker(\widetilde{\xi}\vert_{\ng})$, 
we obtain 
\begin{equation}\label{grad_proof_eq2}
\ng(\widetilde{\xi}\vert_{\ng})\subseteq\widetilde{\ng}(\widetilde{\xi})
\iff [\mg,\ng(\widetilde{\xi}\vert_{\ng})]\subseteq\Ker(\widetilde{\xi}\vert_{\ng}).
\end{equation}
Moreover, for any subset $\Vc\subseteq\widetilde{\ng}$ one has 
$\widetilde{\ng}(\widetilde{\xi})\cap\Vc=\{x\in\Vc\mid [x,\widetilde{\ng}]\subseteq\Ker\widetilde{\xi}\}$, 
so 
\begin{equation}\label{grad_proof_eq3}
\widetilde{\ng}(\widetilde{\xi})\cap\mg\subseteq\mg(\widetilde{\xi}\vert_{\mg})
\text{ and }
\widetilde{\ng}(\widetilde{\xi})\cap\ng\subseteq\ng(\widetilde{\xi}\vert_{\ng}).
\end{equation}
Furthermore, for arbitrary $x\in\widetilde{\ng}$ one has 
\begin{equation}\label{grad_proof_eq4}
x\in\widetilde{\ng}(\widetilde{\xi})\iff [x,\widetilde{\ng}]\subseteq\Ker\widetilde{\xi}\iff 
[x,\mg]\subseteq\Ker\widetilde{\xi} 
\text{ and }[x,\ng]\subseteq\Ker(\widetilde{\xi}\vert_{\ng}). 
\end{equation}

Step 2. 
For proving the proposition, 
we must check that if $[\mg,\ng]\subseteq\Ker\widetilde{\xi}$, then 
$\mg(\widetilde{\xi}\vert_{\mg})\subseteq\widetilde{\ng}(\widetilde{\xi})$ and 
$\ng(\widetilde{\xi}\vert_{\ng})\subseteq\widetilde{\ng}(\widetilde{\xi})$, 
and moreover that 
if $x\in\widetilde{\ng}$, 
$x_{\mg}\in\mg$ and $x_{\ng}\in\ng$ with $x=x_{\mg}+x_{\ng}$, then 
\begin{equation}\label{grad_proof_eq5}
x\in\widetilde{\ng}(\widetilde{\xi})\iff x_{\mg}\in\mg(\widetilde{\xi}\vert_{\mg})
\text{ and }x_{\ng}\in\ng(\widetilde{\xi}\vert_{\ng}).
\end{equation}
Since $[\mg,\ng]\subseteq\ng$, it follows by the hypothesis $[\mg,\ng]\subseteq\Ker\widetilde{\xi}$ 
that actually 
\begin{equation}\label{grad_proof_eq6}
[\mg,\ng]\subseteq\Ker(\widetilde{\xi}\vert_{\ng}).
\end{equation} 
Now, if $x\in\widetilde{\ng}(\widetilde{\xi})$, then we obtain by \eqref{grad_proof_eq4} 
that $[x_{\mg}+x_{\ng},y]\in\Ker(\widetilde{\xi}\vert_{\ng})$ for arbitrary $y\in\ng$. 
Since $[x_{\mg},y]\in[\mg,\ng]\subseteq\Ker(\widetilde{\xi}\vert_{\ng})$ by \eqref{grad_proof_eq6}, 
it then follows that $[x_{\ng},\ng]\subseteq\Ker(\widetilde{\xi}\vert_{\ng})$, 
that is, $x_{\ng}\in\ng(\widetilde{\xi}\vert_{\ng})$. 
This also implies by \eqref{grad_proof_eq6} and \eqref{grad_proof_eq2} 
that $x_{\ng}\in \widetilde{\ng}(\widetilde{\xi})$, 
hence $x_{\mg}=x-x_{\ng}\in \widetilde{\ng}(\widetilde{\xi})$. 
But thus $x_{\mg}\in \widetilde{\ng}(\widetilde{\xi})\cap\mg\subseteq \mg(\widetilde{\xi}\vert_{\mg})$ 
by \eqref{grad_proof_eq3}. 
This completes the proof of the implication ``$\Rightarrow$'' in \eqref{grad_proof_eq5}. 

For the converse implication ``$\Leftarrow$'' in \eqref{grad_proof_eq5}, note that 
$\mg(\widetilde{\xi}\vert_{\mg})\subseteq\widetilde{\ng}(\widetilde{\xi})$ and 
$\ng(\widetilde{\xi}\vert_{\ng})\subseteq\widetilde{\ng}(\widetilde{\xi})$ 
by \eqref{grad_proof_eq6}, \eqref{grad_proof_eq1}, and \eqref{grad_proof_eq2},  
and therefore $x_{\mg},x_{\ng}\in \widetilde{\ng}(\widetilde{\xi})$. 
Since $x=x_{\mg}+x_{\ng}$, it follows that $x\in \widetilde{\ng}(\widetilde{\xi})$, 
and we are done. 
\end{proof}

\begin{remark}\label{P1}
\normalfont 
To point out the representation theoretic significance of the condition from Proposition~\ref{grad}, 
assume there that $\widetilde{\ng}$ is a nilpotent Lie algebra. 
Any unitary irreducible representation $\pi\colon N\to\Bc(\Hc)$ associated with the $N$-coadjoint orbit of $\xi$  
extends to 
a unitary irreducible representation $\widetilde{\pi}\colon\widetilde{N}\to\Bc(\Hc)$ with $\widetilde{\pi}\vert_N=\pi$. 
One can take the representation $\widetilde{\pi}$ associated with the $\widetilde{N}$-coadjoint orbit of $\widetilde{\xi}\in\widetilde{\ng}^*$, 
where $\widetilde{\xi}\vert_\mg=0$ and $\widetilde{\xi}\vert_\ng=\xi$. 
 In fact, one has $\mg\subseteq\widetilde{\ng}(\widetilde{\xi})$, 
hence  the above assertions follow as a very special case of the results of \cite{Du72} 
(see also \cite[Ths. 2.5.1(b), 2.5.3(b)]{CG90}). 
\end{remark}

\begin{remark}\label{SQ}
\normalfont 
One can give an alternative  proof of the square integrability property of stepwise square integrable representations $\pi_\lambda$ from \cite[Th. 3.6]{Wo13a}, by using the above Proposition~\ref{grad} for computing the coadjoint isotropy algebra 
at $\lambda\in\ng^*$ and then applying Proposition~\ref{orthog}. 
\end{remark}

\begin{proposition}\label{jump}
Let $\widetilde{\ng}=\mg\ltimes\ng$, 
where $\widetilde{\ng}$ is a nilpotent Lie algebra. 
Let 
$$\{0\}=\widetilde{\ng}_0\subseteq\widetilde{\ng}_1\subseteq\cdots\subseteq\widetilde{\ng}_m=\widetilde{\ng}$$
be any Jordan-H\"older series with $\widetilde{\ng}_k=\ng$, where $k=\dim\ng$. 
If $\widetilde{\xi}\in\widetilde{\ng}^*$ and  $\widetilde{\ng}(\widetilde{\xi})=\mg(\widetilde{\xi}\vert_{\mg})\dotplus\ng(\widetilde{\xi}\vert_{\ng})$, 
then 
\begin{equation}\label{jump_eq1}
J(\widetilde{\xi})=J(\widetilde{\xi}\vert_\ng)\sqcup (k+J(\widetilde{\xi}\vert_\mg)),
\end{equation} 
where the jump sets 
$J(\widetilde{\xi}\vert_\ng)\subseteq\{1,\dots,k\}$ 
and $J(\widetilde{\xi}\vert_\mg)\subseteq\{1,\dots,m-k\}$ are computed for the Jordan-H\"older series 
$\{0\}=\widetilde{\ng}_0\subseteq\widetilde{\ng}_1\subseteq\cdots\subseteq\widetilde{\ng}_k=\ng$ 
and 
$\{0\}=\widetilde{\ng}_k/\ng\subseteq\widetilde{\ng}_{k+1}/\ng 
\subseteq\cdots\subseteq\widetilde{\ng}_m/\ng=\widetilde{\ng}/\ng\simeq\mg$, 
respectively. 
\end{proposition}

\begin{proof}
Since $\widetilde{\ng}(\widetilde{\xi})=\mg(\widetilde{\xi}\vert_{\mg})\dotplus\ng(\widetilde{\xi}\vert_{\ng})$, 
one has for $j\in\{1,\dots,m\}$,
$$j\in J(\widetilde{\xi})\iff 
\mg(\widetilde{\xi}\vert_{\mg})+\ng(\widetilde{\xi}\vert_{\ng})+\widetilde{\ng}_{j-1}
\subsetneqq
\mg(\widetilde{\xi}\vert_{\mg})+\ng(\widetilde{\xi}\vert_{\ng})+\widetilde{\ng}_j.$$
If $k+1\le j\le m$, then $\ng(\widetilde{\xi}\vert_{\ng})\subseteq\ng=\widetilde{\ng}_k\subseteq\widetilde{\ng}_{j-1}\subseteq\widetilde{\ng}_j$, 
hence 
$$\begin{aligned}
j\in J(\widetilde{\xi})
& 
\iff \mg(\widetilde{\xi}\vert_{\mg})+\widetilde{\ng}_{j-1}
\subsetneqq
\mg(\widetilde{\xi}\vert_{\mg})+\widetilde{\ng}_j \\
& 
\iff (\mg(\widetilde{\xi}\vert_{\mg})+\widetilde{\ng}_{j-1})/\ng
\subsetneqq
(\mg(\widetilde{\xi}\vert_{\mg})+\widetilde{\ng}_j)/\ng \\
&
\iff j-k\in J(\widetilde{\xi}\vert_\mg).
\end{aligned}
$$
If $1\le j\le k$, then $\widetilde{\ng}_{j-1}\subseteq\widetilde{\ng}_j\subseteq\ng$ and $\ng\cap\mg=\{0\}$, hence 
$$\mg(\widetilde{\xi}\vert_{\mg})\cap(\ng(\widetilde{\xi}\vert_{\ng})+\widetilde{\ng}_{j-1})
=
\mg(\widetilde{\xi}\vert_{\mg})\cap(\ng(\widetilde{\xi}\vert_{\ng})+\widetilde{\ng}_j)=\{0\}, $$ 
and this implies 
$$\mg(\widetilde{\xi}\vert_{\mg})+\ng(\widetilde{\xi}\vert_{\ng})+\widetilde{\ng}_{j-1}
\subsetneqq
\mg(\widetilde{\xi}\vert_{\mg})+\ng(\widetilde{\xi}\vert_{\ng})+\widetilde{\ng}_j
\iff \ng(\widetilde{\xi}\vert_{\ng})+\widetilde{\ng}_{j-1}
\subsetneqq
\ng(\widetilde{\xi}\vert_{\ng})+\widetilde{\ng}_j.$$
Therefore 
$j\in J(\widetilde{\xi})\iff 
\ng(\widetilde{\xi}\vert_{\ng})+\widetilde{\ng}_{j-1}
\subsetneqq
\ng(\widetilde{\xi}\vert_{\ng})+\widetilde{\ng}_j 
\iff 
j\in J(\widetilde{\xi}\vert_\ng)$, and 
this completes the proof. 
\end{proof}

\begin{lemma}\label{obv}
Let $\hg$ be any nilpotent Lie algebra with a fixed Jordan-H\"older series, 
and denote by $\zg$ the center of $\hg$. 
If $\dim\zg=1<\dim\hg$ 
and $\Oc\subseteq\hg^*$ is any coadjoint orbit with $e(\Oc)=e(\hg)$, 
then $\langle\Oc,\zg\rangle\ne\{0\}$. 
\end{lemma}

\begin{proof}
 Let $\{0\}=\hg_0\subseteq\hg_1\subseteq\cdots\subseteq\hg_m=\hg$ be the fixed Jordan-H\"older series. 
Since $\hg_1\subseteq\zg$, and $\dim\zg=1$, it follows that $\hg_1=\zg$. 
If $\xi_0\in\hg^*$ is any functional with $\zg\not\subset\Ker\xi_0$, 
then by using \cite[Th. 3.1(2)]{BB13} one can find $X_2\in\hg_2$ with $[X_2,\hg]\not\subset\Ker\xi_0$, 
hence $\hg_2\not\subset\hg(\xi_0)=\hg(\xi_0)+\hg_1$. 
Therefore $2\in J(\xi_0)$ and $1\not\in J(\xi_0)$. 

On the other hand, if $\Oc\subseteq\hg^*$ is any coadjoint orbit with 
$\langle\Oc,\zg\rangle=\{0\}$, then for every $\xi\in\Oc$ one has 
$[\hg_2,\mg]\subseteq\hg_1=\zg\subseteq\Ker\xi$, hence $\hg_1\subseteq\hg_2\subseteq\hg(\xi)$, 
and this shows that $1,2\in J(\xi)$. 
Consequently the coadjoint orbit $\Oc_0$ of $\xi_0$ satisfies 
 $e(\Oc_0)\prec e(\Oc)$, and this shows that $e(\Oc)\neq e(\hg)$. 
\end{proof}

\begin{lemma}\label{interm}
In Proposition~\ref{jump} if $[\mg,\ng]\subseteq\Ker\widetilde{\xi}$, 
$J(\widetilde{\xi})=e(\widetilde{\ng})$ 
and $J(\widetilde{\xi}\vert_\ng)=e(\ng)$, 
then $J(\widetilde{\xi}\vert_{\mg})=e(\mg)$. 
\end{lemma}

\begin{proof}
Propositions \ref{grad}~and~\ref{jump} imply 
$\widetilde{\ng}(\widetilde{\xi})
=\mg(\widetilde{\xi}\vert_{\mg})\dotplus\ng(\widetilde{\xi}\vert_{\ng})$ 
and 
\begin{equation}\label{iter_proof_eq1}
J(\widetilde{\xi})=J(\widetilde{\xi}\vert_{\ng})\sqcup (k+J(\widetilde{\xi}\vert_{\mg})),
\end{equation} 
where $k=\dim\ng$. 
For arbitrary $\zeta\in\mg^*$  
define $\widetilde{\zeta}\in\widetilde{\ng}^*$ by 
$\widetilde{\zeta}\vert_{\mg}:=\zeta$ and 
$\widetilde{\zeta}\vert_{\ng}:=\widetilde{\xi}\vert_{\ng}$. 
Then $[\mg,\ng]\subseteq\Ker\widetilde{\zeta}$  
hence, by Propositions \ref{grad} and \ref{jump} as above, 
one obtains 
$J(\widetilde{\zeta})=J(\widetilde{\zeta}\vert_{\ng})\sqcup (k+J(\widetilde{\zeta}\vert_{\mg}))
=J(\widetilde{\xi}\vert_{\ng})\sqcup (k+J(\zeta))$. 
It then easily follows by \eqref{iter_proof_eq1} 
along with 
$J(\widetilde{\xi})=e(\widetilde{\ng})$ 
that $J(\widetilde{\xi}\vert_{\mg})\preceq J(\zeta)$ 
for all $\zeta\in\mg^*$,  
hence $J(\widetilde{\xi}\vert_{\mg})=e(\mg)$. 
\end{proof}

\begin{lemma}\label{rec}
Let $\widetilde{\ng}=\mg\ltimes\ng$, 
where $\widetilde{\ng}$ is a nilpotent Lie algebra 
and $\mg$ is a Lie algebra with 1-dimensional center and generic flat coadjoint orbits. 
Fix any Jordan-H\"older series of $\widetilde{\ng}$ that passes through~$\ng$ 
and endow $\mg$ and $\ng$ with their corresponding Jordan-H\"older series as in Proposition~\ref{jump}. 
Assume that $\Xc\subseteq\ng^*$ is any subset with the following properties: 
\begin{itemize}
\item The set $\Xc$ is $\Ad_N^*$-invariant and for every $\xi\in\Xc$ one has $J(\xi)=e(\ng)$. 
\item One has a subalgebra $\sg\subseteq\ng$ and a linear subspace $\Vc\subseteq\ng$ 
with $\ng=\sg\dotplus\Vc$, $[\mg,\sg]=\{0\}$, $[\mg,\ng]\subseteq\Vc$, and 
for every $N$-coadjoint orbit $\Oc\subseteq\Xc$ 
there exists $\eta\in\Oc$ with $\Vc\subseteq\Ker\eta$ 
and $\ng(\eta)=\sg$. 
\end{itemize}
Let $\widetilde{\Xc}\subseteq\widetilde{\ng}^*$ be any $\Ad_{\widetilde{N}}^*$-invariant subset 
with $J(\widetilde{\xi})=e(\widetilde{\ng})$ and $\widetilde{\xi}\vert_{\ng}\in\Xc$ for all $\widetilde{\xi}\in\widetilde{\Xc}$. 
Then for every $\widetilde{N}$-coadjoint orbit $\widetilde{\Oc}\subseteq\widetilde{\Xc}$ 
and every hyperplane $\Vc_0\subset\mg$ with $\Zc(\mg)+\Vc_0=\mg$ 
there exists $\widetilde{\xi}_0\in\widetilde{\Oc}$  
with $\Vc_0+\Vc\subseteq\Ker\widetilde{\xi}_0$, $\mg(\widetilde{\xi}_0\vert_{\mg})=\Zc(\mg)$, 
$\ng(\widetilde{\xi}_0\vert_{\ng})=\sg$, 
and $\widetilde{\ng}(\widetilde{\xi}_0)
=\Zc(\mg)\dotplus\sg$. 
\end{lemma}

\begin{proof}
Let $\widetilde{\xi}\in\widetilde{\Oc}$ arbitrary. 
One has $\widetilde{\xi}\vert_{\ng}\in\Xc$, 
hence by using the hypothesis for the $N$-coadjoint orbit $\Ad^*_N(N)(\widetilde{\xi}\vert_{\ng})\subseteq\Xc$  
one finds $y\in\ng$ with 
$
\Vc\subseteq\Ker\eta$ and $\ng(\eta)=\sg$, 
where $\eta:=\Ad^*_N(y)(\widetilde{\xi}\vert_{\ng})$. 
Since $[\widetilde{\ng},\ng]\subseteq\ng$, 
one has $\Ad^*_N(y)(\widetilde{\xi}\vert_{\ng})=(\Ad^*_{\widetilde{N}}(y)\widetilde{\xi})\vert_{\ng}$. 
One thus finds $\widetilde{\xi}_1:=\Ad^*_{\widetilde{N}}(y)\widetilde{\xi}\in\widetilde{\Oc}$ 
with 
$
\Vc\subseteq\Ker\widetilde{\xi}_1$, 
and hence $[\mg,\ng]\subseteq\Ker\widetilde{\xi}_1$. 

One has $\widetilde{\xi}_1\in\widetilde{\Oc}\subseteq\widetilde{\Xc}$, 
so  
$\widetilde{\xi}_1\vert_{\ng}\in\Xc$ and $J(\widetilde{\xi}_1\vert_{\ng})=e(\ng)$ 
by hypothesis. 
Then $J(\widetilde{\xi}_1\vert_{\mg})=e(\mg)$ by Lemma~\ref{interm}. 
Since $\mg$ has generic flat coadjoint orbits, 
one then obtains 
$\Ker(\Ad^*_M(x)(\widetilde{\xi}_1\vert_{\mg}))=\Vc_0$ for some $x\in\mg$,  
with $\Vc_0\subset\mg$ as in the statement.  
Since $[\mg,\mg]\subseteq\mg$ and $[\mg,\ng]\subseteq\ng$, 
it follows that 
\begin{equation}\label{rec_proof_eq2}
\Ad_{\widetilde{N}}(x)=
\begin{pmatrix}
\Ad_M(x) & 0 \\
0 & \Ad_{\widetilde{N}}(x)\vert_{\ng}
\end{pmatrix} 
\end{equation}
with respect to the vector direct sum $\widetilde{\ng}=\mg\dotplus\ng$. 
In particular $\Ad^*_M(x)(\widetilde{\xi}_1\vert_{\mg})=(\Ad^*_{\widetilde{N}}(x)\widetilde{\xi}_1)\vert_{\mg}$, 
and we thus find $\widetilde{\xi}_0:=\Ad^*_{\widetilde{N}}(x)\widetilde{\xi}_1\in\widetilde{\Oc}$ 
with $\mg(\widetilde{\xi}_0\vert_{\mg})=\Zc(\mg)$ and $\Vc_0\subseteq\Ker\widetilde{\xi_0}$. 

Also, 
$$\begin{aligned}
\Ker(\widetilde{\xi}_0)
& 
=\Ker(\widetilde{\xi}_1\circ\Ad_{\widetilde{N}}(-x))
=\Ad_{\widetilde{N}}(x)(\Ker\widetilde{\xi}_1) \\
& 
\supseteq \Ad_{\widetilde{N}}(x)\Vc=\Vc\supseteq[\mg,\ng],
\end{aligned}
$$ 
where the equality $\Ad_{\widetilde{N}}(x)\Vc=\Vc$ follows by \eqref{rec_proof_eq2}  
since $x\in\mg$ and $[\mg,\Vc]\subseteq\Vc\subseteq\ng$.  
Now $\widetilde{\ng}(\widetilde{\xi}_0)
=\mg(\widetilde{\xi}_0\vert_{\mg})\dotplus\ng(\widetilde{\xi}_0\vert_{\ng})$ by Proposition~\ref{grad}. 
In addition, one has  
$$\begin{aligned}
\ng(\widetilde{\xi}_0\vert_{\ng})
&
=\ng((\Ad^*_{\widetilde{N}}(x)\widetilde{\xi}_1)\vert_{\ng}) 
=\ng(\Ad^*_{\widetilde{N}}(x)(\widetilde{\xi}_1\vert_{\ng}))
= \Ad_{\widetilde{N}}(x)(\ng(\widetilde{\xi}_1\vert_{\ng})) \\
&
=\Ad_{\widetilde{N}}(x)(\ng(\eta))
=\Ad_{\widetilde{N}}(x)\sg
=\sg
\end{aligned} 
$$
where the latter equality follows since $[\mg,\sg]=\{0\}$ and $x\in\mg$, 
and this completes the proof. 
\end{proof}

\begin{theorem}\label{main1}
Let $\ng$ be a nilpotent Lie algebra with a fixed Jordan-H\"older basis, 
and a distinguished subset of the corresponding Jordan-H\"older series, denoted  
$\{0\}=\ng_0\subseteq\ng_1\subseteq\cdots\subseteq\ng_q=\ng$. 
Assume that for every $j=1,\dots,q$ one has a subalgebra $\mg_j\subseteq\ng_j$ 
and a linear subspace $\Vc_j\subseteq\mg_j$ compatible with the above fixed basis, 
satisfying the following conditions: 
\begin{itemize}
\item The center $\zg_j$ of $\mg_j$ is 1-dimensional, $\mg_j=\zg_j\dotplus \Vc_j$, 
$\mg_j$ has flat generic coadjoint orbits, 
and one has the semidirect product decomposition $\ng_j=\mg_j\ltimes\ng_{j-1}$. 
\item One has $[\mg_j,\ng_{j-1}]\subseteq\Vc_1\dotplus\cdots\dotplus\Vc_{j-1}
$ 
and $[\mg_j,\zg_1+\cdots+\zg_{j-1}]=\{0\}$. 
\end{itemize}
Let 
$\Xc:=\{\xi\in\ng^*\mid J(\xi\vert_{\ng_j})=e(\ng_j)\text{ for }j=1,\dots,q\}$ 
and $\sg:=\zg_1+\cdots+\zg_q$.
Then one has: 
\begin{enumerate}
\item The set $\Xc$ is $\Ad^*_N$-invariant. 
\item For any $N$-coadjoint orbit $\Oc\subseteq\ng^*$ one has 
$\Oc\subseteq\Xc$ if and only if there exists  
$\xi\in\Oc$ with $\ng(\xi)=\sg$ 
and  
$
\Vc_1\dotplus\cdots\dotplus\Vc_q\subseteq\Ker\xi$, 
and then $\xi$ is uniquely determined by these properties.  
\end{enumerate}
\end{theorem}

\begin{proof}
The fact that $\Xc$ is $\Ad^*_N$-invariant follows by \cite[page 426]{Pd84}. 
Then the proof proceeds by induction on $q\ge1$. 
If $q=1$, then $\ng=\ng_1=\mg_1$, and the conclusion is clear. 

Now assume the assertion was proved for some value of $q$, 
and use the notation from the statement. 
To perform the induction step, 
let $\widetilde{\ng}=\mg\ltimes\ng$ be a nilpotent Lie algebra 
with a fixed Jordan-H\"older series that contains the Jordan-H\"older series of $\ng$, 
where $\mg=\zg\dotplus\Vc_0$ is a Lie algebra with 1-dimensional $\zg$ and generic flat coadjoint orbits, 
and assume $[\mg,\ng]\subseteq\Vc_1\dotplus\cdots\dotplus\Vc_q$ and $[\mg,\sg]=\{0\}$, 
where $\sg:=\zg_1+\cdots+\zg_q$. 
Define $\Xc$ as in the statement and let 
$\widetilde{\Xc}:=\{\widetilde{\xi}\in\widetilde{\ng}^*\mid J(\widetilde{\xi})=e(\widetilde{\ng})
\text{ and }\widetilde{\xi}\vert_{\ng}\in\Xc\}$, 
that is, $\widetilde{\Xc}$ is the set associated to $\widetilde{\ng}$ 
in the same way as $\Xc$ is associated to $\ng$ in the statement of the theorem. 
Then Lemma~\ref{rec} implies that every $\widetilde{N}$-coadjoint orbit $\widetilde{\Oc}\subseteq\widetilde{\Xc}$ 
has the properties we are looking for. 
Conversely, select 
any $\widetilde{N}$-coadjoint orbit $\widetilde{\Oc}\subseteq\widetilde{\ng}^*$ 
for which there exists $\widetilde{\xi}\in\widetilde{\Oc}$ with $\widetilde{\ng}(\widetilde{\xi})=\Zc(\mg)+\zg_1+\cdots+\zg_q$, 
$\Vc_1\dotplus\cdots\dotplus\Vc_q\dotplus\Vc_0\subseteq\Ker\xi$.   
Then Proposition~\ref{grad} implies ${\ng}(\widetilde{\xi})=\mg(\widetilde{\xi}\vert_{\mg})\dotplus\ng(\widetilde{\xi}\vert_{\ng})$, 
and $\ng(\widetilde{\xi}\vert_{\ng})=\zg_1+\cdots+\zg_q$, 
hence $\widetilde{\xi}\vert_{\ng}\in\Xc$ by the induction hypothesis. 
Moreover, one can prove by using Proposition~\ref{jump} that $J(\widetilde{\xi})=e(\widetilde{\ng})$.  
Therefore $\widetilde{\xi}\in\widetilde{\Xc}$, hence $\widetilde{\Oc}\subseteq\widetilde{\Xc}$.   

Finally, for proving uniqueness of $\xi\in\Oc\subseteq\Xc$ in Assertion~(ii),   
consider the jump set $e:=e(\widetilde{\Oc})$ and note that  
$\Vc_1\dotplus\cdots\dotplus\Vc_q=\ng_e$ (see the notation in Preliminaries). 
But the restriction map $\Oc\to\ng_e^*$, $\xi\mapsto\xi\vert_{\ng_e}$ is a diffeomorphism, 
hence there exists a unique $\xi\in\Oc$ with $\xi\vert_{\ng_e}=0$, 
and this completes the proof. 
\end{proof}

\begin{proposition}\label{main2}
In the framework of Theorem~\ref{main1}, 
if in addition $q=2$ and $\dim\Zc(\ng)=1$, then $\Xc=\{\xi\in\ng^*\mid J(\xi)=e(\ng)\}$. 
\end{proposition}

\begin{proof}
By definition, 
$\Xc:=\{\xi\in\ng^*\mid J(\xi\vert_{\ng_j})=e(\ng_j)\text{ for }j=1,2\}$, 
so we must prove that if $\xi\in\ng^*$ and $J(\xi)=e(\ng)$, 
then $J(\xi\vert_{\ng_1})=e(\ng_1)$. 

In fact, if $J(\xi)=e(\ng)$, then by Lemma~\ref{obv} one has $\Zc(\ng)\cap\Ker\xi=\{0\}$. 
On the other hand, since $\ng=\mg_1\ltimes\ng_1$ and $[\mg_1,\zg_1]=\{0\}$, 
it follows that $\zg_1\subseteq\Zc(\ng)$.
Therefore $\zg_1\cap\Ker(\xi\vert_{\ng_1})=\{0\}$, 
hence $J(\xi\vert_{\ng_1})=e(\ng_1)$ since $\ng_1$ has 1-dimensional center and generic flat coadjoint orbits. 
\end{proof}

\begin{corollary}\label{main3}
Assume the setting of Theorem~\ref{main1}. 
If $\pi\colon N\to\Bc(\Hc)$ is any unitary irreducible representation associated with the coadjoint orbit $\Oc$ 
of some $\xi\in\ng^*$ with $(\zg_1\cup\cdots\cup\zg_q)\cap\Ker\xi=\{0\}$  
and  
$
\Vc_1\dotplus\cdots\dotplus\Vc_q\subseteq\Ker\xi$, 
then $\Oc\subseteq\Xc$ and 
$\Vert(\pi(\cdot)f\mid h)\Vert_{L^2(N/S)}^2=C\Vert f\Vert^2\Vert h\Vert^2$ for all $f,h\in\Hc$, 
where $C>0$ is a constant depending on~$\Oc$.
\end{corollary}

\begin{proof}
It is easily seen that 
$\ng(\xi)=\sg$, using Proposition~\ref{grad}.  
Then the assertion follows by Theorem~\ref{main1} and Proposition~\ref{orthog}. 
\end{proof}

\begin{corollary}\label{main4}
Assume the setting of Theorem~\ref{main1} and let $[\Xc]\subseteq\widehat{N}$ 
be the subset of the unitary dual of~$N$ corresponding to the $N$-coadjoint orbits contained in~$\Xc$.  
Then the interior of $[\Xc]$ is dense in $\widehat{N}$, $[\Xc]$ is Hausdorff in its relative hull-kernel topology, 
and the Plancherel measure of $\widehat{N}\setminus[\Xc]$ is zero. 
\end{corollary}

\begin{proof}
Let $\Xc_0$ (respectively, $\Xc_1$) be the lowest layer in the fine (respectively, coarse) 
stratification of $\ng^*$ 
corresponding to the fixed Jordan-H\"older series in~$\ng$. 
We recall from \cite{Pd89} that $\Xc_0$ is defined similarly to $\Xc$ 
except that one uses the whole Jordan-H\"older series and not only its subset $\{\ng_j\mid j=1,\dots,q\}$, 
while $\Xc_1:=\{\xi\in\ng^*\mid J(\xi)=e(\ng)\}$, 
hence $\Xc_0\subseteq\Xc\subseteq\Xc_1$. 
If $[\Xc_j]$ is the subset of $\widehat{N}$ corresponding to the $N$-coadjoint orbits contained in~$\Xc_j$,   
then both $[\Xc_0]$ and $[\Xc_1]$ are open, dense, and Hausdorff in their relative hull-kernel topologies 
(see \cite{Pd84}), 
hence $[\Xc]$ has dense interior and is Hausdorff. 
Finally, the Plancherel measure of $\widehat{N}\setminus[\Xc_0]$ is zero \cite{Pu67}. 
\end{proof}

\subsubsection*{Examples}

\begin{example}
\normalfont
It follows by the proof of \cite[Lemma 6.5]{Wo13a} that the above Theorem~\ref{main1} 
applies to any nilpotent Lie algebra~$\ng$ 
that is the nilradical of a minimal Borel subalgebra  
of a split real form of of a complex semisimple Lie algebra, 
and in this case each $\mg_r$ is a Heisenberg algebra. 
\end{example}

\begin{example}
\normalfont 
If $\ng$ is any 3-step nilpotent Lie algebra with 1-dimensional center denoted by~$\zg$, 
then by \cite[Th. 5.1--5.2]{BB13} there exists a semidirect product decomposition 
$\widetilde{\ng}=\ag\ltimes\mg$, 
where $\ag$ is an abelian Lie algebra, $\mg$ is a 3-step nilpotent Lie algebra with its center equal to~$\zg$ and with generic flat coadjoint orbits, 
and there exists a linear subspace $\Vc\subseteq\mg$ with $\mg=\zg\dotplus\Vc$, 
$[\ag,\Vc]\subseteq\Vc$, and $[\ag,\zg]=\{0\}$. 
Hence by writing $\ag$ as a direct sum of 1-dimensional abelian algebras, one sees that $\ng$ is a Lie algebra to which Theorem~\ref{main1} applies. 
\end{example}

\section*{Acknowledgements}
This research has been partially supported by the Grant
of the Romanian National Authority for Scientific Research, CNCS-UEFISCDI,
project number PN-II-ID-PCE-2011-3-0131.

\end{document}